\documentclass{amsart}

\usepackage{mathscinet}

\usepackage[all]{xy}
\usepackage{graphics,bm}
\usepackage{color}
\usepackage[T1]{fontenc}
\usepackage{parskip}
\usepackage[colorlinks]{hyperref}
\usepackage{enumitem}
\usepackage{comment}
\usepackage{mathtools}

\usepackage{tikz,tikz-cd}
\usetikzlibrary{decorations.markings}

\usepackage[colorinlistoftodos,prependcaption]{todonotes}

\newtheorem{theorem}{Theorem}[section]
\newtheorem{theorem*}{Theorem}
\newtheorem{lemma}[theorem]{Lemma}
\newtheorem{proposition}[theorem]{Proposition}

\theoremstyle{definition}

\newtheorem*{remark*}{Remark}
\newtheorem*{remarks*}{Remarks}
\newtheorem*{corollary*}{Corollary}

\numberwithin{figure}{section}
\numberwithin{table}{section}
\numberwithin{equation}{section}

\newcommand{\OP}{\operatorname}

\begin{document}

\title{A short proof that the $L^p$-diameter of $Diff_0(S,area)$ is infinite.}
\author{Micha\l\ Marcinkowski}
\address{Institute of Mathematics, Polish Academy of Sciences, Wroc\l aw, Poland}
\email{marcinkow@math.uni.wroc.pl}
\thanks{The author was supported by grant Sonatina 2018/28/C/ST1/00542 funded by Narodowe Centrum Nauki} 
%MSC 51 (geometry) 53 (diff geometry) 57 (manifolds and cell complexes)
%57K10: "Knot theory" (MSC2020)
%58D05: "Groups of diffeomorphisms and homeomorphisms as manifolds" (MSC2020)
%37E30: "Dynamical systems involving homeomorphisms and diffeomorphisms of planes and surfaces" (MSC2020) 

\begin{abstract} 
We give a short proof that the $L^p$-diameter of the group of area preserving diffeomorphisms isotopic to the identity of a compact surface is infinite.
\end{abstract}

\maketitle
 
\section{Introduction}

Let $(M,g)$ be a Riemannian manifold and let $\mu$ be the measure induced by the metric $g$. 
By $\OP{Diff}_0(M,\mu)$ we denote the group of all diffeomorphisms of $M$ that preserve $\mu$ and are isotopic to the identity.

In \cite{shni.rus} A. Shnirelman showed that the $L^2$-diameter of $\OP{Diff}_0(M,\mu)$
is finite if $M$ is the $n$-dimensional ball for $n>2$ (see also \cite{shni.gafa}). 
Conjecturally the same is true as well for any compact simply connected Riemannian manifold of dimension greater than $2$ (it is stated in \cite{eliashberg.ratiu} without proof).

The situation is different for $2$-dimensional manifolds.
In this case it is customary to denote the measure induced by $g$ by $\OP{area}$. 
For simplicity, let us restrict the discussion to orientable compact connected Riemannian surfaces $(S,g)$.
Eliashberg and Ratiu \cite{eliashberg.ratiu} proved that the $L^p$-diameter ($p\geq 1$) of $\OP{Diff}_0(S,\OP{area})$ 
is infinite if $S$ is a surface with boundary. 
They show that the Calabi homomorphism is Lipschitz with respect to the $L^p$-norm.
Later Gambaudo and Lagrange \cite{gambaudo.lagrange} obtained a similar result for a huge class of quasimorphisms on 
$\OP{Diff}_0(S,\OP{area})$ if $S$ is the closed disc (see as well \cite{brandenburskylpmetrics,brandshelaut,egoraction} for more 
results concerning quasimorphisms and the $L^p$ geometry).
Their proof makes use of the braid group of the disc and inequalities relating the geometric intersection number of a braid and its word-length.  

If $S$ has negative Euler characteristic, it is relatively easy to show that the $L^p$-diameter for $p \geq 1$ of $\OP{Diff}_0(S,\OP{area})$ is infinite 
(see Proposition \ref{p:basic.argument} or \cite[Theorem 1.2]{brandy.kedra}).
In the case of the torus one needs to know in addition that the group of hamiltonian diffeomorphisms 
of the torus is simply connected, which is a non-trivial result from symplectic topology (see \cite[Appendix A]{brandy.shelukhin}).
 
The last unsolved case was the sphere. Recently Brandenbursky and Shelukhin \cite{brandy.shelukhin} showed that in this 
case the diameter is as well infinite.
Moreover, for each $p \geq 1$, $\OP{Diff}_0(S^2,\OP{area})$ contains quasi-isometrically embedded right-angled Artin groups \cite{kim.koberda} 
and $\mathbb{R}^m$ for each natural $m$. 
Their arguments use some new tools along with the ideas from \cite{gambaudo.lagrange}. However, 
using intersection numbers in the case of the sphere requires considerably more work. 

The aim of this paper is to give a short and elementary proof of the following theorem.

\begin{theorem*}
Let $(S,g)$ be a compact surface (with or without boundary). Then for every $p\geq 1$
the $L^p$-diameter of $\OP{Diff}_0(S,area)$ is infinite.
\end{theorem*}

Our method gives a unified proof for every compact surface $S$.
It is partially inspired by \cite{gambaudo.lagrange}, in particular Lemma \ref{l:embedding} can be seen as a
generalization of an inequality obtained in \cite{gambaudo.lagrange} for the disk. 
The main simplification comes from the fact that instead of using the braid group and intersection numbers, 
we directly look at the geometry of the configuration space $C_n(S)$ with a certain complete metric described in Section~\ref{s:complete}.
In Section \ref{s:embedding} we relate the $L^1$-norm of $f \in \OP{Diff}_0(S,\OP{area})$ 
to a $L^1$-norm, defined by this complete metric, of the diffeomorphism on $C_n(S)$ induced by~$f$. 
This allows us to apply a simple technique, described in Section \ref{s:basic}, 
of showing the unboundedness of the $L^p$-norm in the case where the fundamental group of the manifold is complicated enough. 

\section{The $L^p$-norm}

Let $(M,g)$ be a Riemannian manifold and let $\mu$ be a finite measure on $M$.
Usually one assumes that $\mu$ is induced by $g$, even though the definition of $L^p$-norm works as well if $\mu$ is any finite measure
(then the $L^p$-norm could be a pseudo-norm). 
We introduce here a more general definition as it is useful for stating results in Section~\ref{s:embedding}.

Suppose $f \in \OP{Diff}_0(M,\mu)$ and let $X \colon M \to TM$ be a map to a tangent space of $M$ such that $X(x) \in T_{f(x)}M$. 
One can think of $X$ as a tangent vector to $\OP{Diff}_0(M,\mu)$ at the point $f$. 
The $L^p$-norm of $X$ is defined by the formula
$$\|X\|_p = \Big(\int_M |X(x)|^p dx\Big)^{\frac{1}{p}}.$$

Let $f_t \in \OP{Diff}_0(M,\mu)$, $t \in [0,1]$, be a smooth isotopy, i.e., it defines a smooth map $M \times [0,1] \to M$.
We always assume that isotopies are smooth.
The $L^p$-length of $\{f_t\}$ is defined by
$$l_p(\{f_t\}) = \int_0^1 \|\dot{f}_t\|_p dt, $$

where $\dot{f}_t(x) = \frac{d}{ds} f_s(x)|_{s=t} \in T_{f_t(x)}M$. 
Note that if $p=1$, then $\int_0^1 |\dot{f}_t(x)|dt$ is the length of the path $f_t(x)$, thus
$l_1(\{f_t\})$ can be interpreted as the $\mu$-average of the lengths of all paths $f_t(x)$.

Let $f \in \OP{Diff}_0(M,\mu)$, we define the $L^p$-norm of $f$ by
$$l_p(f) = \inf l_p(\{f_t\}),$$
where the infimum is taken over all smooth isotopies $f_t \in \OP{Diff}_0(M,\mu)$ connecting the identity on $M$ with $f$.
The assumption that $f$ is $\mu$-preserving was not used in the definition, but it is needed to show that $l_p$ satisfies the triangle inequality.

The $L^p$-diameter of $\OP{Diff}_0(M,\mu)$ equals 
$$\sup\{l_p(f)~\colon~f \in \OP{Diff}_0(M,\mu)\}.$$

It is worth noting that geodesics in $\OP{Diff}_0(M,\mu)$ with the $L^2$-metric are solutions of the Euler equations of an incompressible fluid. 
For more on the connection between the $L^2$-metric and hydrodynamics see \cite{arnold.khesin}.

\section{The base case}\label{s:basic}

In this section we present the basic method which can be used to show that for $p\geq 1$, 
the $L^p$-diameter of $\OP{Diff}_0(M,\mu)$ is infinite if $\pi_1(M)$ is complicated enough.
We start with a lemma.

\begin{lemma}\label{l:center}
Let $X$ be a topological space and let $f_t \in \OP{Homeo}(X)$, $t\in [0,1]$, be a loop in $\OP{Homeo}(X)$ based at $Id_X$, i.e., 
$f_0 = f_1 = Id_X$. Then for every $x \in X$, the loop $f_t(x)$, $t\in [0,1]$, is in the center of $\pi_1(X,x)$.  
\end{lemma}

\begin{proof}
Let $x \in X$ and let $\gamma_s$, $s \in [0,1]$, be a loop in $X$ based at $x$. 
Consider the map $\phi \colon S^1 \times S^1 \to X$ given by $(t,s) \to f_t(\gamma_s)$,
where $S^1 = [0,1]/_{0 \sim 1}$. We have that $\phi(t,0) = f_t(x)$ and $\phi(0,s) = \gamma_s$.
Thus loops $f_t(x)$ and $\gamma_s$ are in the image of the torus $S^1 \times S^1$, therefore they commute.  
\end{proof}

Let $(M,g)$ be a Riemannian manifold.
Suppose $h \in \pi_1(M)$. Let $l(h)$ denote the infimum over lengths of based loops in $M$ that represent $h$. 
We denote by $Z(\pi_1(M))$ the center of $\pi_1(M)$. 

\begin{proposition}\label{p:basic.argument}
Let $(M,g)$ be a Riemannian manifold and $\mu$ the measure induced by $g$. 
Assume that for every $r$ the set $\{h \in \pi_1(M) \colon l(h) < r\}$ is finite (it holds e.g., if $M$ is compact) 
and $\pi_1(M)/Z(\pi_1(M))$ is infinite. 	
Then for every $p\geq 1$ the $L^p$-diameter of $\OP{Diff}_0(M,\mu)$ is infinite. 
\end{proposition}

\begin{proof}
By the H\"older inequality we can assume $p=1$.
Let $z \in M$ be a base-point and let $h \in \pi_1(M,z)$.
We represent $h$ as a loop $\gamma$ based at $z$. 
%such that $l(\gamma) \leq l(h)+1$, where $l(\gamma)$ is the length of $\gamma$.

Let $U$ be a contractible neighborhood of $z$ and let $f_t \in \OP{Diff}_0(M,\mu)$, $t\in [0,1]$, 
be a finger-pushing isotopy that moves $U$ all the way along $\gamma$. For detailed construction see \cite[Proof of Lemma 3.1]{mb}.

For every $x \in U$ we choose a path $\phi_x$ contained in $U$ connecting $z$ with $x$. 
We can assume that $l(\phi_x) < \OP{diam}(U)$, where $l(\phi_x)$ is the length of $\phi_x$. We denote by $\phi^*_x$ the reverse of $\phi_x$. 

The isotopy $f_t$ is defined such that it satisfies the following properties: 
\begin{enumerate}
\item for every $x \in U$, $f_1(x) = x$.
\item for every $x \in U$, the concatenation of $\phi_x$, $f_t(x)$ and $\phi^*_x$ is a loop based at $z$ and its homotopy class equals $h$.   
\end{enumerate}

Let $f_h = f_1$ and define $L_h = \min\{ l(hc) \colon c \in Z(\pi_1(M,z) \}$. 
We shall show that 
$$\mu(U)(L_h - 2\OP{diam}(U)) \leq l_1(f_h).$$

Let $g_t$, $t \in [0,1]$, be any isotopy connecting the identity on $M$ with $f_h$. 
Due to Lemma \ref{l:center}, for every $x \in U$, the paths $g_t(x)$ and $f_t(x)$ represent elements of $\pi_1(M,x)$ that differ by an element of the center. 
Thus the concatenation of $\phi_x$, $g_t(x)$ and $\phi^*_x$ represents an element of the form $hc \in \pi_1(M,z)$ where $c \in Z(\pi_1(M,z))$. 
Since $l(\phi_x) < \OP{diam}(U)$, we have that $l(g_t(x)) \geq L_h - 2\OP{diam}(U)$. 
Indeed, otherwise the concatenation of $\phi_x$, $g_t(x)$ and $\phi^*_x$ would be a loop of length less then $L_h \leq l(hc)$,
which is impossible. 

Since $l(g_t(x)) = \int_0^1|\dot{g}_t(x)|dt$, we have 

\begin{align*}
	\mu(U)(L_h - 2\OP{diam}(U)) &\leq \int_{U} \int_0^1 |\dot{g}_t(x)| dt dx&\\
	&\leq \int_M \int_0^1 |\dot{g}_t(x)| dt dx&\\
	&= l_1(\{g_t\}).&
\end{align*}

The isotopy $g_t$ was arbitrary, therefore $\mu(U)(L_h - 2\OP{diam}(U)) \leq l_1(f_h)$. 

By assumption, for every $r$ the set $S_h = \{h \in \pi_1(M) \colon l(h) < r\}$ is finite. 
Therefore, since $\pi_1(M)/Z(\pi_1(M))$ is infinite, there exists $h$ such that the coset $hZ(\pi_1(M))$ does not intersect $S_h$.  
For such $h$ we have $L_h \geq r$. 
Since the set $U$ does not depend of the choice of $h$, and $L_h$ can be arbitrary large, 
we conclude that the $L^1$-diameter of $\OP{Diff}_0(M,\mu)$ is infinite. 

\end{proof}

In particular, Proposition \ref{p:basic.argument} can be applied when $(S,g)$ is a compact surface of negative Euler characteristic
(then $\pi_1(S)$ is infinite and has trivial center).
Unfortunately, it says nothing about the $L^p$-diameter of $\OP{Diff}_0(S,\OP{area})$ for the remaining surfaces. 
The main goal of this paper is to find an argument which is still based on the proof of Proposition \ref{p:basic.argument},
but works for any compact surface $S$.

To this end one could pass to the configuration space of $n$ ordered points in $S$, denoted $C_n(S) \subset S^n$,
with the product Riemannian metric $g^n$.
Its fundamental group is the pure braid group $P_n(S)$, and $P_n(S)/Z(P_n(S))$ is infinite for every $S$ if $n>3$.  
However, the problem with this space is that every braid $P_n(S)$ can be represented as a based loop in $(C_n(S),g^n)$ of length
at most $2n\OP{diam}(S)+1$, thus one cannot apply Proposition \ref{p:basic.argument}. 

We solve this problem by changing the metric on $C_n(S)$.
We describe it, in a slightly more general setting, in the next section. 

\section{A complete metric on a manifold with removed submanifolds}\label{s:complete}

Let $(M, g)$ be a compact Riemannian manifold and let $D = \cup_{i=1}^{k} D_i$, where $D_i$ are submanifolds of~$M$.
The aim of this paragraph is to construct a metric on $M - D$ satisfying the following property: for every $L$ the number of elements in $\pi_1(M-D)$ that can be represented by a based loop of length less then $L$ is finite. 
For $x \in M$, denote by $d(x)$ the distance of $x$ to $D$, i.e.,

$$d(x) = d_g(x,D) = min\{d_g(x,D_i) : i=1,\ldots,k\},$$

where $d_g$ is the metric on $M$ induced by $g$.  

We rescale $g$ by $\frac{1}{d}$,
i.e., we define a new quadratic form $g_b$ on the tangent space of $M-D$ by
$$|v|_{g_b} = \frac{|v|_g}{d(x)},$$
where $v \in T_x (M-D)$ is a vector tangent to a point $x \in M-D$. 

Note that $d(x)$, and consequently $g_b$, are not differentiable.
They are only continuous. In this case $g_b$ is called a $C^0$-Riemannian metric and
a smooth manifold with such a quadratic form is called a $C^0$-Riemannian manifold. 
A $C^0$-Riemannian structure allows to define lengths of paths
and a metric $d$ on the underlying manifold. 
The topology induced by $d$ is equal to the manifold topology.

\begin{lemma}\label{l:compl}
$M-D$ with the metric $g_b$ is a complete $C^0$-Riemannian manifold.
\end{lemma}

\begin{proof}

Let $N = (M-D,g_b)$ and let $B_N(x,r)$ denote the closed ball in $N$ of radius $r$ and center $x \in N$.
To show completeness it amounts to show that for every $x \in N$ the ball $B_N(x,\frac{1}{2})$ is compact. 

Let $x \in N$. We shall show that the distance of $B_N(x,\frac{1}{2})$ to $D$ is at least $\frac{d(x)}{2}$,
i.e., 
$$B_N(x,\frac{1}{2}) \subset L \coloneqq \{y \in N \colon d(y)\geq \frac{d(x)}{2}\}.$$
Since $L$ is compact, it follows that $B_N(x,\frac{1}{2})$ is compact. 

Suppose $y \in B_N(x,\frac{1}{2})$ and $d(y) < d(x)$ (otherwise obviously $y \in L$). 
Let $\epsilon > 0$ and let $\gamma \colon [0,l] \to N$ be a path connecting $x$ with $y$ such that 
$|\dot{\gamma}(t)|_{g_b} = 1$ for $t \in [0,l]$ and $l < \frac{1}{2} + \epsilon$.

Let 
$$t_0  = \sup \{ t \in [0,l]~\colon~d(\gamma(t)) \geq d(x)\},$$
i.e., $t_0$ is the last time when $d(\gamma(t_0)) = d(x)$.  
For $t \geq t_0$, we have 
$$|\dot{\gamma}(t)|_g = |\dot{\gamma}(t)|_{g_b}d(\gamma(t)) = d(\gamma(t)) \leq d(x).$$
Let $\gamma'$ be the restriction of $\gamma$ to the interval $[t_0,l]$.
Let $l_g(\gamma')$ be the length of $\gamma'$ in the metric $g$. Since $|\dot{\gamma}(t)|_g \leq d(x)$, we have  
$$l_g(\gamma') \leq (l-t_0)d(x) \leq (\frac{1}{2}+\epsilon)d(x).$$ 
Therefore the distance of $y$ to $D$ in $g$ is at least 
$$d(y) \geq d(\gamma(t_0)) - l_g(\gamma') \geq d(x) - (\frac{1}{2}+\epsilon)d(x) = \frac{d(x)}{2}-\epsilon d(x).$$ 
Since $\epsilon$ is arbitrarily small, $y \in L$ and therefore $B_N(x,\frac{1}{2}) \subset L$. 

\end{proof}

Before we proceed we need the following simple lemma. 
Note that this lemma would be standard if $(M-D,g_b)$ were a complete Riemannian manifold.

\begin{lemma}\label{l:compact.ball}
Let $N = (M-D,g_b)$ and let $\widetilde{N}$ be the universal cover of $N$ with the pulled-back $C^0$-Riemannian metric. 
Then every closed ball in $\widetilde{N}$ is compact.
\end{lemma}

\begin{proof}
By the Weierstrass approximation theorem, there exists $C \in \mathbb{R}$ and a smooth function $f \colon N \to \mathbb{R}$ such that
$C^{-1}f(x) < 1/d(x) < Cf(x)$ for every $x \in N$. Let $g_s$ be a Riemannian metric defined by $|v|_{g_s} = f(x)|v|_g$, where $v \in T_xN$.
Then $C^{-1}|v|_{g_s} < |v|_{g_b} < C|v|_{g_s}$, 
thus the metrics induced by $g_b$ and $g_s$ are equivalent.
By Lemma \ref{l:compl}, $(N,g_s)$ is a complete Riemannian manifold
and it is a standard fact that closed balls in the universal cover of $(N,g_s)$ are compact.  
Clearly, it holds as well for $(N,g_b)$, since the metrics defined by pull-backs of $g_s$ and $g_b$ to the universal cover are equivalent.  

\end{proof}

Let $h \in \pi_1(M-D)$. Denote by $l(h)$ the infimum of lengths (with respect to $g_b$) of based loops representing $h \in \pi_1(M-D)$.

\begin{lemma}\label{l:finite.set} 
%Let $N$ be a complete smooth manifold with a $C^0$-Riemannian structure. 
For every $r$, the set $\{h \in \pi_1(M-D) \colon l(h) < r\}$ is finite. 
\end{lemma}

\begin{proof}
Let $N = (M-D,g_b)$, let $x \in N$ be a basepoint and let $p \colon \widetilde{N} \to N$ be the universal cover of $N$.
Choose $y \in p^{-1}(x)$. 
The pre-image $p^{-1}(x)$ is discrete and $B_{\widetilde{N}}(y,r) \subset \widetilde{N}$ is compact by Lemma \ref{l:compact.ball}.
Thus $p^{-1}(x) \cap B_{\widetilde{N}}(y,r)$ is finite for every~$r$ and therefore $\{h \in \pi_1(N) \colon l(h) < r\}$ is finite.
\end{proof}

\section{A Lipschitz embedding}\label{s:embedding}
In this section we focus on the particular case where $M-D$ is a configuration space. 
Let $(S,g)$ be a compact Riemannian surface and
$g^n$ be the product metric on $S^n$. Let $D_{ij} = \{ (x_1,\ldots,x_n) \in S^n \colon x_i = x_j \}$. Denote by $C_n(S) = S^n - \cup_{i,j} D_{ij}$
the configuration space of $n$ ordered points in $S$. 
On $S^n$ and $C_n(S)$ we consider the measure induced by the product metric $g^n$. 

We shall now find a formula for $d_{g^n}(x,D_{ij})$ in terms of the metric on $S$. 
Let $x=(x_1,\ldots,x_n)~\in~S^n$ and let $m$ be the midpoint of a geodesic connecting $x_i$ with $x_j$.
If we start moving points $x_i$ and $x_j$ towards $m$ with constant speed, we get a geodesic in $S^n$ connecting $x$ with the closest point in $D_{ij}$.
Since $d_g(m,x_i) = d_g(m,x_j) = \frac{1}{2}d_g(x_i,x_j)$ and we are in the product metric, we get 
$$d_{g^n}(x,D_{ij}) = \sqrt{d_g(m,x_i)^2 + d_g(m,x_j)^2} = \frac{1}{\sqrt{2}}d_g(x_i,x_j).$$ 
 
The distance function $d$ has the form 
$$
d(x) = \frac{1}{\sqrt{2}}min \{ d_g(x_i,x_j) \colon 1 \leq i < j \leq n\}. 
$$

Let $g_b = (g^n)_b$ be the metric on $C_n(S)$ defined in the previous section, namely
$|v|_{g_b} = \frac{|v|_{g^n}}{d(x)}$, where $v \in T_x(C_n(S))$.  

Let us fix a point $p \in S$ and let $x=(x_1,\ldots,x_{n-1}) \in S^{n-1}$. 
Then $(p,x) \in S^n$ and $d((p,x))$ is the minimum over $\frac{1}{\sqrt{2}}d_g(p,x_i)$ for $1\leq i \leq n-1$ 
and $\frac{1}{\sqrt{2}}d_g(x_i,x_j)$ for $1 \leq i < j \leq n-1$. 

We need the following technical lemma.

\begin{lemma}\label{l:finite.int}
There exists $C \in \mathbb{R}$ such that for every $p \in S$ we have

$$
	\int_{S^{n-1}} \frac{1}{d((p,x))}dx \leq C.
$$
\end{lemma}

\begin{proof}
	It can be easily seen using polar coordinates that there exists $C'$ such that for every $p \in D^2$, where $D^2$ is the euclidean disc, we have
	$$\int_D \frac{1}{|p-x|}dx < C'.$$

	Since such $C'$ exists for a disc, we as well have a similar bound for every compact surface $S$, i.e., for every $p \in S$ we have
	$$\int_S \frac{1}{d_g(p,x)}dx < C'.$$
	
	After integrating over all possible $p \in S$ we have (we assume $\OP{area}(S)=1$)
	$$\int_{S^2} \frac{1}{d_g(p,x)}dpdx < C'.$$
	
	Let  $x = (x_1,\ldots,x_{n-1})$. Since $d((p,x))$ is the minimum over $\frac{1}{\sqrt{2}}d_g(p,x_i)$, $i~=~1,\ldots,n~-~1$,
	and $\frac{1}{\sqrt{2}}d_g(x_i,x_j)$, $1 \leq i < j \leq n-1$, we have that
	$$
	\frac{1}{d((p,x))} \leq \sum_{i} \frac{\sqrt{2}}{d_g(p,x_i)} + \sum_{i \neq j} \frac{\sqrt{2}}{d_g(x_i,x_j)}.
	$$
	Thus 
	\begin{align*}
		&\int_{S^{n-1}} \frac{1}{d((p,x))}dx \leq  \\ 
		&\leq \sum_{i} \int_{S^{n-1}} \frac{\sqrt{2}}{d_g(p,x_i)}dx + \sum_{i \neq j} \int_{S^{n-1}} \frac{\sqrt{2}}{d_g(x_i,x_j)}dx= \\
		&= (n-1)\int_{S} \frac{\sqrt{2}}{d_g(p,x)}dx + \frac{n(n-1)}{2} \int_{S^{2}} \frac{\sqrt{2}}{d_g(x_1,x_2)}dx_1dx_2 \leq\\
		&\leq \sqrt{2}(n-1)C'+ \frac{n(n-1)}{\sqrt{2}}C' \eqqcolon C.
	\end{align*}
\end{proof}

Let $\mu$ be the measure on $C_n(S)$ induced by the product metric $g^n$.
A diffeomorphism $f \in \OP{Diff}_0(S,\OP{area})$ defines a product diffeomorphism $f_* \in \OP{Diff}_0(C_n(S),\mu)$.
Namely, for $x = (x_1,\ldots,x_n) \in S^n$ we have $f_*(x) = (f(x_1),\ldots,f(x_n))$.
Thus we have a product embedding $\OP{Diff}_0(S,\OP{area}) \hookrightarrow \OP{Diff}_0(C_n(S),\mu)$.

On $\OP{Diff}_0(C_n(S),\mu)$ we consider the $L^1$-norm defined by the metric $g_b$ and the measure $\mu$.
Note that here we are in the case where $g_b$ and $\mu$ are not compatible, that is, the measure
induced by $g_b$ and the measure $\mu$ are different. 

The following lemma provides a link between the $L^1$-norm on $\OP{Diff}_0(S,\OP{area})$ 
and the $L^1$-norm on $\OP{Diff}_0(C_n(S),\mu)$ defined above.
Note that in the proof it is essential that $f$ preserves the area on $S$.

\begin{lemma}\label{l:embedding}
The product embedding $\OP{Diff}_0(S,\OP{area}) \hookrightarrow \OP{Diff}_0(C_n(S),\mu)$
is Lipschitz, i.e., there exists $C$ such that $l_1(f_*) \leq Cl_1(f)$. 
\end{lemma}

\begin{proof}

Let $f \in \OP{Diff}_0(S,\OP{area})$ and let $X \colon S \to TS$ such that $X(x) \in T_{f(x)}S$. 
For $x = (x_1,\ldots,x_n)~\in~C_n(S)$ we define $X_*(x) = (X(x_1),\ldots,X(x_n))\in T_{f_*(x)}C_n(S)$. 

The set $\cup_{i,j} D_{ij} \subset S^n$ is of measure zero.
It means that we can regard $|X_*(x)|_{g_b}$ as a measurable function defined on $S^n$.
Thus in what follows, we integrate $|X_*(x)|_{g_b}$ over $S^n$ with the product measure rather then over $C_n(S)$. 

To prove the lemma, it is enough to show that there exists $C$ such that for every 
$f~\in~\OP{Diff}_0(S,\OP{area})$ and every map $X \colon S \to TS$ such that $X(x) \in T_{f(x)}S$ the following inequality holds
$$
	\|X_*\|_1 \leq C\|X\|_1.
$$

Recall that by definition $\|X_*\|_1 = \int_{S^n}|X_*(x)|_{g_b} dx$. We have

\begin{align*}
	\int_{S^n}|X_*(x)|_{g_b} dx &= \int_{S^n} \frac{|X_*(x)|_{g^n}}{d(f_*(x))}dx\\
	&=\int_{S^n} \frac{\sqrt{|X(x_1)|^2_g+\ldots+|X(x_n)|^2_g}}{d(f_*(x))}dx\\
	& \leq \int_{S^n} \frac{|X(x_1)|_g+\ldots+|X(x_n)|_g}{d(f_*(x))}dx\\
	& = n \int_{S^n} \frac{|X(x_1)|_g}{d(f_*(x))}dx.\\
\end{align*}

Since $f_*$ preserves the measure on $S^n$, we have 

\begin{align*}
	\int_{S^n} \frac{|X(x_1)|_g}{d(f_*(x))}dx &= \int_{S^n} \frac{|X \circ f^{-1}(x_1)|_g}{d(x)}dx&\\
	& = \int_S |X \circ f^{-1}(x_1)|_g \Big(\int_{S^{n-1}} \frac{1}{d(x_1,x)}dx\Big) dx_1&\\
	& \leq C \int_S |X \circ f^{-1}(x_1)|_g dx_1& \text{Lemma \ref{l:finite.int}}\\
	& = C \int_S |X(x_1)|_g dx_1&\\ 
	& = C\|X\|_1.&
\end{align*}

\end{proof}

\section{Proof of the theorem}
\setcounter{theorem*}{0}
\begin{theorem*}
Let $(S,g)$ be a compact surface (with or without boundary). Then for every $p\geq 1$
the $L^p$-diameter of $\OP{Diff}_0(S,area)$ is infinite.
\end{theorem*}

\begin{proof}
By the H\"older inequality we can assume $p=1$. Fix $n>3$. 

Let $z~=~(z_1,\ldots,z_n)~\in~C_n(S)$. Denote the pure braid group on $n$ strings by $P_n(S) = \pi_1(C_n(S),z)$. 
Suppose $U_i \subset S$ are disjoint discs such that $z_i \in U_i$.
Let $U = U_1 \times U_2 \ldots \times U_n \subset C_n(S)$.

Choose $h \in P_n(S)$ and $\gamma$ a loop in $C_n(S)$ representing $h$. 
Let $f_t \in \OP{Diff}_0(S,\OP{area})$, $t\in [0,1]$, be an isotopy such that $(f_t)_* \in \OP{Diff}_0(C_n(S),\mu)$ moves $U$ all the way along $\gamma$
and have properties (1) and (2) from the proof of Proposition \ref{p:basic.argument}. Let $f_h = f_1$. 

It is convenient to imagine that $f_t$ moves $U_i$ along the trajectory of $z_i$ given by~$\gamma$. 
In fact, to construct $f_t$ satisfying the above properties for a general $h \in P_n(S)$, it is enough to do it for a given finite set of generators 
of $P_n(S)$ (or generators of the full braid group $B_n(S)$). In \cite{bellingeri} one can find a set of generators of $B_n(S)$, for which
the construction of $f_t$ is straightforward.  

Recall that on $C_n(S)$ we consider the complete metric $g_b$. 
By Lemma \ref{l:finite.set}, the set $\{h \in \pi_1(C_n(S)) \colon l(h) < r\}$ is finite for every $r$
and $P_n(S)/Z(P_n(S))$ is infinite.
It follows from the proof of Proposition \ref{p:basic.argument} that $l_1((f_h)_*)$ can be arbitrary large.

Therefore, due to Lemma \ref{l:embedding}, $l_1(f_h)$ can be arbitrary large. 
Thus the $L^1$-diameter of $\OP{Diff}_0(S,\OP{area})$ is infinite. 

\end{proof}

%\nocite{*}
\bibliography{bibliography}
\bibliographystyle{plain}

\end{document}